\documentclass[runningheads,11pt]{llncs}

\usepackage[T1]{fontenc}
\usepackage[english]{babel}
\usepackage[utf8]{inputenc}
\usepackage{comment}

\usepackage{geometry}
\usepackage{amsmath,amssymb,amsfonts}
\usepackage{mathtools}
\usepackage{chngcntr}

\numberwithin{theorem}{section}
\numberwithin{lemma}{section}
\numberwithin{corollary}{section}
\spnewtheorem{claim2}{Claim}{\bfseries}{\itshape}  
\numberwithin{claim2}{section}
\spnewtheorem{fact}{Fact}{\bfseries}{\itshape}  
\numberwithin{fact}{section}

\usepackage{scalerel}
\usepackage{stackengine,amsmath}
\stackMath

\usepackage[dvipsnames]{xcolor}
\usepackage{graphicx}
\usepackage{subcaption}
\usepackage{standalone}
\usepackage{tikz}
\usetikzlibrary{arrows,snakes,backgrounds,patterns,matrix,shapes,fit,calc,shadows,plotmarks,arrows.meta,positioning,decorations.pathmorphing,decorations.pathreplacing,decorations.markings,fixedpointarithmetic}
\usepackage[noend]{algpseudocode}
\usepackage{algorithm}

\usepackage{eqparbox}
\algnewcommand\LeftComment[1]{%
$\triangleright$ \eqparbox{COMMENT}{#1} \hfill %
}

\usepackage{hyperref}
\usepackage{url}
\hypersetup{%
    backref = page,%
    colorlinks,%
    linkcolor = {red!60!black},%
    citecolor = {green!60!black},%
    urlcolor = {blue!60!black}%
}

\usepackage{doi}

\usepackage{enumerate}   
\usepackage{paralist}
\usepackage{lineno}

\newcommand{\CD}{\mathrm{CD}}
\newcommand{\NP}{\textsf{NP}}

\newcommand{\defi}[1]{\emph{\color{red!50!black}#1}}

\usepackage{stackengine,amsmath}
\stackMath 

\def\shrinkage{-2.4mu}
\def\vecsign#1{\rule[1.38\LMex]{\dimexpr#1-6.35pt}{0.45\LMpt}\kern-3.0\LMpt\mathchar"017E}
\def\dvecsign#1{\rule{0pt}{8\LMpt}\smash{\stackon[-2.15\LMpt]{\SavedStyle\mkern-\shrinkage\vecsign{#1}}%
  {\rotatebox{180}{$\SavedStyle\mkern-\shrinkage\vecsign{#1}$}}}}
\def\dvec#1{\ThisStyle{\setbox0=\hbox{$\SavedStyle#1$}%
  \def\useanchorwidth{T}\stackon[-4.7\LMpt]{\SavedStyle#1}{\dvecsign{\wd0}}}}

\def\cevsign#1{\rule[1.38\LMex]{\dimexpr#1-6.3pt}{0.45\LMpt}\kern-3.0\LMpt\mathchar"017E}
\def\lvecsign#1{\rule{0pt}{8\LMpt}\smash{{\reflectbox{$\SavedStyle\mkern-\shrinkage\cevsign{#1}$}}}}
\def\lvec#1{\ThisStyle{\setbox0=\hbox{$\SavedStyle#1$}%
  \def\useanchorwidth{T}\stackon[-4.5\LMpt]{\SavedStyle#1}{\,\lvecsign{\wd0}}}}
\def\lveclow#1{\ThisStyle{\setbox0=\hbox{$\SavedStyle#1$}%
  \def\useanchorwidth{T}\stackon[-6\LMpt]{\SavedStyle#1}{\,\lvecsign{\wd0}}}}

\def\oldvec{\mathaccent"017E}
\newcommand{\rarrow}[1]{\oldvec{#1}}
\newcommand{\larrow}[1]{\lvec{#1}}
\newcommand{\larrowlow}[1]{\lveclow{#1}}
\newcommand{\lrarrow}[1]{\dvec{#1}}
\def\dichi{\dvec{\chi}}
\def\diclique{\dvec{\omega}}
\def\clique{\omega}

\begin{document}

\title{A study on token digraphs}

\author{
  Cristina G.\ Fernandes\inst{1}\orcidID{0000-0002-5259-2859}
  \and
  Carla N.\ Lintzmayer\inst{2}\orcidID{0000-0003-0602-6298} 
  \and
  Juan~P.~Peña\inst{3}
  \and
  Giovanne Santos\inst{3}
  \and 
  Ana Trujillo-Negrete\inst{4}\orcidID{0000-0002-1138-1190} 
  \and 
  Jose Zamora\inst{5}
}

\authorrunning{Fernandes, Lintzmayer, Peña, Santos, Trujillo, Zamora}

\institute{
  Departamento de Ciência da Computação.
  Universidade de São Paulo.
  Brazil\\
  \email{\href{mailto:cris@ime.usp.br}{cris@ime.usp.br}}
  \and
  Centro de Matemática, Computação e Cognição.
  Universidade Federal do ABC.
  Brazil\\
  \email{\href{mailto:carla.negri@ufabc.edu.br}{carla.negri@ufabc.edu.br}}
  \and
  Departamento de Ingeniería Matemática.
  Universidad de Chile.
  Chile\\
  \email{\href{mailto:juanpabloalcayaga.p21@gmail.com}{juanpabloalcayaga.p21@gmail.com}},
  \email{\href{mailto:gsantos@dim.uchile.cl}{gsantos@dim.uchile.cl}}
  \and
  Center for Mathematical Modeling.
  University of Chile.
  Chile\\
  \email{\href{mailto:ltrujillo@dim.uchile.cl}{ltrujillo@dim.uchile.cl}}
  \and
  Departamento de Matemáticas.
  Universidad Andres Bello.
  Chile\\
  \email{\href{mailto:josezamora@unab.cl}{josezamora@unab.cl}}
}

\maketitle

\begin{abstract}
  For a digraph~$D$ of order~$n$ and an integer $1 \leq k \leq n-1$, the
  \defi{$k$-token digraph} of~$D$ is the graph whose vertices are all
  $k$-subsets of vertices of~$D$ and, given two such $k$-subsets~$A$ and~$B$,
  $(A,B)$ is an arc in the $k$-token digraph whenever $\{a\} = A \setminus B$,
  $\{b\} = B \setminus A$, and there is an arc $(a,b)$ in~$D$.
  Token digraphs are a generalization of token graphs.
  In this paper, we study some properties of token digraphs, including strong
  and unilateral connectivity, kernels, girth, circumference and Eulerianity.
  We also extend some known results on the clique and chromatic numbers of
  $k$-token graphs, addressing the bidirected clique number and dichromatic
  number of $k$-token digraphs.
  Additionally, we prove that determining whether $2$-token digraphs have a
  kernel is \NP-complete. 
\end{abstract}


\section{Introduction}

Let~$G$ be a simple graph of order $n \geq 2$ and let~$k$ be an integer, with
$1 \leq k \leq n-1$.
The \defi{$k$-token graph of~$G$}, denoted by~$F_k(G)$, is the graph whose
vertices are all the $k$-subsets of~$V(G)$, two of which are adjacent if their
symmetric difference is a pair of adjacent vertices in~$G$. 

The name ``token graph'' is motivated by the following 
interpretation~\cite{FabilaMonroyFHHUW2012}.
Consider a graph~$G$ and a fixed number~$k$ of indistinguishable tokens, with
$1 \leq k \leq n-1$.
A \defi{$k$-token configuration} corresponds to the~$k$ tokens placed on
distinct vertices of~$G$, which corresponds to a subset of~$k$ vertices of~$G$. 
Tokens can slide from their current vertex to an unoccupied adjacent vertex
and, at each step, exactly one token can slide.
Construct a graph whose vertices are the $k$-token configurations, and make two
such configurations adjacent whenever one configuration can be reached from the
other by sliding a token from one vertex to an adjacent vertex.
This new graph is isomorphic to~$F_k(G)$.
See an example in Figure~\ref{fig:example_tokengraph}.

\begin{figure}[h!]
    \centering
    \resizebox{0.6\textwidth}{!}{\includegraphics{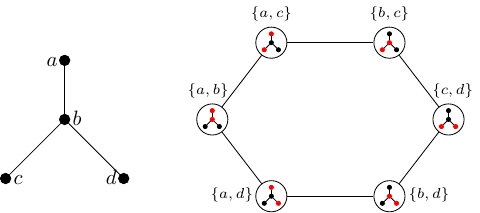}}
    \caption{A graph~$G$ and its $2$-token graph~$F_2(G)$.}
    \label{fig:example_tokengraph}
\end{figure}

The $k$-token graphs are also called the \defi{symmetric $k$th power} of
graphs~\cite{AudenaertGRR2007}.
In fact, token graphs have been defined, independently, at least four times
since~1988 (see~\cite{double_vertex,FabilaMonroyFHHUW2012,Johns,Rudolph2002}). 
They have been used to study the Isomorphism Problem in graphs by means of
their spectra, and, so far, some applications of token graphs to Physics and
Coding Theory are known~\cite{hamiltonian-cycles,EtzionB1996,packing,Rudolph2002}.
Also, token graphs are related with other well-known graphs, 
such as the Johnson graphs and the doubled Johnson graphs.
See, e.g.,~\cite{AlaviLL2002,EtzionB1996}. 

In 2012, Fabila-Monroy \textit{et al.}~\cite{FabilaMonroyFHHUW2012}
reintroduced $k$-token graphs and proved tight lower and upper bounds on their
diameter, connectivity, an chromatic number.
They also characterized the cliques in token graphs in terms of the cliques in
the original graph and established sufficient conditions for the existence or
non-existence of a Hamiltonian path in various token graphs.
They showed that if~$F_k(G)$ is bipartite for some $k \geq 1$, then~$F_\ell(G)$
is bipartite for all $\ell \geq 1$.
Carballosa \textit{et al.}~\cite{CarballosaFLR2017} then characterized, for
each value of~$k$, which graphs have a regular $k$-token graph and which
connected graphs have a planar $k$-token graph.
Also, de Alba \textit{et al.}~\cite{deAlbaCR2020} presented a tight lower bound
for the matching number of~$F_k(G)$ for the case in which~$G$ has either a
perfect matching or an almost perfect matching, estimated the independence
number for bipartite $k$-token graphs, and determined the exact value for some
graphs.
Adame \textit{et al.}~\cite{AdameRT2021} provided an infinite family of graphs,
containing Hamiltonian and non-Hamiltonian graphs, for which their $k$-token
graphs are Hamiltonian.

In this paper, we consider a generalization of token graphs to digraphs as
follows.
Let $D$ be a digraph of order~$n$ and let~$k$ be an integer with $1 \leq k \leq
n-1$.
The \defi{$k$-token digraph of~$D$}, denoted by~$F_k(D)$, is the digraph whose
vertices are all the $k$-subsets of~$V(D)$ and, given two such $k$-subsets~$A$
and~$B$, $(A,B)$ is an arc in~$F_k(D)$ whenever $\{a\} = A \setminus B$, 
$\{b\} = B \setminus A$, and there is an arc $(a,b)$ in~$D$.
See Figure~\ref{fig:example_tokendigraph} for an example.

\begin{figure}[h!]
    \centering
    \resizebox{0.6\textwidth}{!}{\includegraphics{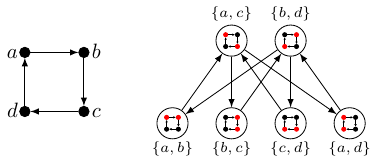}}
    \caption{A digraph $D$ and its $2$-token digraph $F_2(D)$.}
    \label{fig:example_tokendigraph}
\end{figure}

As far as we know, the only generalization of token graphs to digraphs prior to our work was
proposed by Gao and Shao~\cite{GaoS2009} in 2009 for $k=2$ tokens, but with the
following difference: they considered 2-tuples of vertices instead of
2-subsets. 
This corresponds to distinct tokens. 
In the final stage of our work, we became aware of a master's thesis by 
Fernández-Velázquez~\cite{Fernandez2022} that independently introduces the concept of token
digraphs as proposed here. 
In~\cite{Fernandez2022}, the author studies some invariants of token digraphs,
including strong connectivity, diregularity, diameter and the existence of 
oriented cycles, and also explores the game of cops and robbers on token digraphs of
 certain families of graphs, such as cycles, trees and Cartesian products.  

Section~\ref{sec:basics} presents some notation and states some basic
properties we will use throughout the paper. 
In Section~\ref{sec:strong_connectivity}, we present results regarding strong
connectivity of token digraphs while in Section~\ref{sec:kernels} we present
some results regarding kernels of token digraphs.
Section~\ref{sec:uni} considers unilateral digraphs, and characterizes when
their token digraphs also are unilateral. 
In Section~\ref{sec:girth}, we analyze relation between the girth and
circumference of a digraph and their value for its token digraphs and, in
Section~\ref{sec:eulerian}, we study whether the fact that a digraph is
Eulerian or Hamiltonian implies that its token digraphs are also Eulerian or
Hamiltonian, respectively. 
Section~\ref{sec:clique} determines some relations between the bidirected
clique number of a digraph and the bidirected clique number of its token
digraphs. 
Finally, Section~\ref{sec:acyclic_partition} studies acyclic partitions, and
the dichromatic number of a digraph and its token digraphs. 
We conclude with some final remarks in Section~\ref{sec:final}.

\section{General definitions and basic properties}
\label{sec:basics}

We denote the complete graph on~$n$ vertices by~$K_n$, the wheel graph on~$n+1$
vertices by~$W_n$, and the cycle graph on~$n$ vertices by~$C_n$. 

For a digraph~$D$, we denote by~$V(D)$ its vertex set and by~$A(D)$ its arc
set.
For an arc $(x,y) \in A(D)$, we say~$x$ is an \defi{in-neighbor} of~$y$
while~$y$ is an \defi{out-neighbor} of~$x$.
The \defi{in-degree} of a vertex~$v$ is the number of in-neighbors of~$v$,
denoted by $d^-(v)$, and the \defi{out-degree} of~$v$ is the number of
out-neighbors of~$v$, denoted by $d^+(v)$. 

An \defi{oriented path} of~$D$ is a sequence $(v_1,v_2,\ldots,v_k)$ of vertices
of~$D$ such that $v_i \neq v_j$ for all $i \neq j$ and $(v_i,v_{i+1}) \in A(D)$
for all $1 \leq i < k$.
An \defi{$xy$-path} in~$D$ is an oriented path from vertex~$x$ to vertex~$y$.
An \defi{oriented cycle} is a sequence $(v_1,v_2,\ldots,v_k,v_{k+1})$ of
vertices of~$D$ such that $(v_1,v_2,\ldots,v_k)$ is an oriented path,
$(v_k,v_1) \in A(D)$, and $v_{k+1} = v_1$.
The \defi{length} of an oriented path or cycle is the number of edges it
contains.
A \defi{$k$-cycle} or a \defi{$k$-path} is a cycle or a path of length~$k$.
A \defi{digon} is an oriented 2-cycle of~$D$.
We say a cycle or a path is odd (resp.\ even) if its length is odd (resp.\
even).

For a graph~$G$, we denote by $\lrarrow{G}$ the digraph obtained from~$G$ by
replacing each edge of~$G$ by a digon.  
For a digraph~$D$, the \defi{reverse} of~$D$, denoted as~$\larrow{D}$, is the
digraph obtained from~$D$ by reversing every arc of~$D$. 

A digraph is \defi{acyclic} if it has no oriented cycle. 
An acyclic digraph is referred to as a \defi{dag} (acronym for directed acyclic
graph).
A vertex of a dag~$D$ with out-degree zero is called a \defi{sink}.
It is known that every dag has at least one sink.

A \defi{$k$-token configuration} is a $k$-subset of vertices of~$D$, and the
vertices of the token digraph~$F_k(D)$ are $k$-token configurations of~$D$. 
In several of our proofs, we will use the interpretation of an arc~$XY$
of~$F_k(D)$ existing when one can slide or move one token from~$X$ to~$Y$.
Also, in order to avoid confusion, we will use \emph{nodes} to refer to
vertices of token digraphs.


It is easy to see that, for every graph~$G$ on~$n$ vertices, the token
graph~$F_k(G)$ is isomorphic to the token graph~$F_{n-k}(G)$. 
In particular, if~$n$ is even, this gives a non-trivial automorphism
on~$F_{n/2}(G)$. 
The following properties hold.

\begin{property}\label{property1}
For a digragh $D$ on $n$ vertices, the $k$-token digraph $F_k(D)$ is isomorphic
to the $(n-k)$-token digraph $F_{n-k}(\larrow{D})$.
\end{property}

\begin{property}\label{property2}
For a digragh $D$, the $k$-token digraphs $F_k(D)$ and $\larrowlow{F_k(\larrow{D})}$
are isomorphic. 
\end{property}

\begin{property}\label{property3}
For a graph $G$, the $k$-token digraphs $F_k(\lrarrow{G})$ and
$\lrarrow{F_k(G)}$ are isomorphic. 
\end{property}

\section{Strong connectivity aspects of token digraphs}
\label{sec:strong_connectivity}

Let~$D$ be a digraph.
Note that~$D$ is weakly connected if and only if~$F_k(D)$ is weakly connected.
This comes from the fact that the token graph of a graph~$G$ is connected if
and only if~$G$ is connected~\cite[Theorems~5 and~6]{FabilaMonroyFHHUW2012}. 
We want to strengthen this statement somehow to strong connectivity.  

A digraph~$D$ is \defi{strongly connected} if, for every two vertices~$x$
and~$y$ of~$D$, there is an $xy$-path and a $yx$-path.
A digraph that is not strongly connected consists of some \defi{strongly
connected components}, which are the maximal strongly connected subdigraphs
of~$D$.
Thus, to understand the general situation, it is important to consider what
happens in strongly connected digraphs. 

\begin{lemma}
\label{lem:strong_directedpath}
  Let~$D$ be a strongly connected digraph, and let~$A$ and~$B$ be two nodes
  of~$F_k(D)$.
  There is an $AB$-path in~$F_k(D)$. 
\end{lemma}  
\begin{proof}
  Recall that~$A$ and~$B$ are $k$-subsets of $V(D)$. 
  The proof is by induction on $|A \setminus B|$. 
  Clearly if $|A \setminus B|=0$, then $A=B$ and there is nothing to prove. 
  Suppose $|A \setminus B| > 0$. 
  Observe that this implies $|A \cap B| \leq k-1$ and, consequently, $|B
  \setminus A| > 0$, since $|B| = k$.
  Let~$P$ be a shortest oriented path in~$D$ from a vertex in $A \setminus B$
  to a vertex in~$B \setminus A$.
  Call~$a$ the starting vertex of~$P$ and~$b$ the ending vertex of~$P$. 
  The goal is to move the token from~$A$ in~$a$ to~$b$ going through the
  vertices in~$P$.
  However there might be tokens of~$A$ and~$B$ on the way. 
  Let $u_1,\ldots,u_t$ be the vertices on~$P$, in order, that are in $A \cap
  B$.
  Let $u_0=a$ and $u_{t+1}=b$. 
  For $i=t,t-1,\ldots,0$, we move the token from~$A$ in~$u_i$ to~$u_{i+1}$
  through the vertices of~$P$. 
  The resulting $k$-token configuration is $A' = A \setminus \{a\} \cup \{b\}$,
  and the previous process describes an oriented path from~$A$ to~$A'$
  in~$F_k(D)$. 
  Now, as $|A' \setminus B| = |A \setminus B| - 1$, by induction, there is an
  oriented path in~$F_k(D)$ from~$A'$ to~$B$.
  These two paths together contain a path from~$A$ to~$B$ in~$F_k(D)$. 
\hfill$\qed$
\end{proof}  

The previous result alone implies that, if~$D$ is strongly connected,
then~$F_k(D)$ is also strongly connected.
We now address digraphs that are not necessarily strongly connected. 

Let $C_1,\ldots,C_t$ be the strongly connected components of~$D$. 
The \defi{condensation digraph} of~$D$, denoted by $\CD(D)$, is 
obtained from~$D$ by collapsing all vertices in each~$C_i$ to a 
single vertex~\cite{SedgewickW2011}, which we also denote as~$C_i$.
It is known that the condensation digraph of any digraph is a dag.
Hence we may assume without loss of generality that $C_1,\ldots,C_t$ is a
topological order of the vertices of~$\CD(D)$, that is, all arcs in $\CD(D)$ go
from a~$C_i$ to a~$C_j$ with $i < j$. 

Let~$D_1$ and~$D_2$ be two digraphs.
The \defi{Cartesian product} of~$D_1$ and~$D_2$ is the digraph whose vertex set
is the Cartesian product $V(D_1) \times V(D_2)$, and there is an arc between
$(u_1,u_2)$ and $(v_1,v_2)$ if and only if $u_1v_1 \in A(D_1)$ and $u_2 = v_2$,
or $u_2v_2 \in A(D_2)$ and $u_1 = v_1$. 
In what follows, we abuse notation and treat a strongly connected component of
a digraph as its vertex set. 

\begin{lemma}
\label{lem:node_set_cartesian}
  Let~$D$ be a digraph and $C_1,\ldots,C_t$ be a topological order of the
  strongly connected components of~$D$. 
  Let~$A$ be a node of~$F_k(D)$ and let~$k_j = |A \cap C_j|$ for $j =
  1,\ldots,t$.
  Then the strongly connected component of~$F_k(D)$ containing~$A$ is isomorphic
  to the Cartesian product of~$F_{k_j}(C_j)$ for every $j$ with $k_j > 0$. 
\end{lemma}
\begin{proof}
  Recall that nodes of~$F_k(D)$ are $k$-token configurations of~$D$. 
  So~$A$ is a $k$-token configuration of~$D$. 
  Let~$B$ be an arbitrary $k$-token configuration with exactly~$k_j$ tokens in
  each~$C_j$.
  Then~$B$ is a node in the Cartesian product of~$F_{k_j}(C_j)$ for every~$j$
  with $k_j > 0$. 
  Let us argue that~$A$ and~$B$ are in the same strongly connected component
  of~$F_k(D)$.
  For $i = 0,\ldots, t$, let~$A_i$ be the $k$-token configuration that
  coincides with~$B$ in $C_1,\ldots,C_i$ and coincides with~$A$ in
  $C_{i+1},\ldots,C_t$.
  Note that $A_0=A$ and $A_t=B$. 
  Similarly to what we did in Lemma~\ref{lem:strong_directedpath}, for
  each~$i$, there is an oriented path from~$A_{i-1}$ to~$A_i$ that only moves
  tokens within~$C_i$.
  By concatenating these oriented paths, one obtains an $AB$-path.
  Similarly, one can deduce that there is a $BA$-path.
  So~$A$ and~$B$ are in the same strongly connected component of~$F_k(D)$. 

  On the other hand, let~$B$ be a node of~$F_k(D)$ with $|A \cap C| \neq |B
  \cap C|$ for some strongly connected component~$C$ of~$D$.
  Then~$B$ is not in the Cartesian product of~$F_{k_j}(C_j)$ for every $j$
  with~$k_j > 0$.
  Let us argue that, in this case, $A$ and~$B$ are in different strongly
  connected component of~$F_k(D)$.
  Let~$C$ be the first strongly connected component, in the order, such that
  $|A \cap C| \neq |B \cap C|$.
  Assume that $|A \cap C| < |B \cap C|$.  (The other case is symmetric.) 
  Let us prove that there is no $AB$-path.
  Observe that a token can only move inside the same strongly connected
  component of~$D$, or to a component that appears ahead in the order of the
  strongly connected components.
  So, because~$A$ has more tokens than~$B$ in components after~$C$ in the
  order, and none of these tokens can move to~$C$ or to components before~$C$
  in the order, there is no way to move the tokens that started in~$A$ so that
  they finish in~$B$. 
  That is, there is no $AB$-path in~$F_k(D)$.
  Hence~$A$ and~$B$ cannot be in the same strongly connected component
  of~$F_k(D)$.

  Therefore, the node set of the strongly connected component of~$F_k(D)$
  containing~$A$ is isomorphic to the node set of the Cartesian product of
  $F_{k_j}(C_j)$ for every $j$ with $k_j>0$. 
  Moreover, let~$B_1$ and~$B_2$ be two nodes in the strongly connected
  component of~$F_k(D)$ containing~$A$.
  Then~$B_1$ and~$B_2$ are in the Cartesian product of $F_{k_j}(C_j)$ for
  every~$j$ with $k_j>0$.
  There is an arc from~$B_1$ to~$B_2$ in~$F_k(D)$ if and only if, for some $j
  \in \{1,\ldots,t\}$ with $k_j > 0$, there is an arc $(u,v) \in A(D)$ with $u,
  v \in C_j$ so that $B_1 = B_2 \setminus \{u\} \cup \{v\}$.
  And this holds if and only if there is an arc in the Cartesian product
  of~$F_{k_j}(C_j)$ for every~$j$ with $k_j>0$. 
\hfill$\qed$
\end{proof}

Let~$D$ be a digraph and $C_1,\ldots,C_t$ be a topological order of the
strongly connected components of~$D$, as in Lemma~\ref{lem:node_set_cartesian}. 
For each integer vector $(k_1,\ldots,k_t)$ such that 
$0 \leq k_i \leq |C_i|$ for every~$i$ and $\sum_{i=1}^t k_i = k$, 
there is a strongly connected component of~$F_k(D)$ isomorphic to 
the Cartesian product of~$F_{k_j}(C_j)$ for every $j$ with $k_j > 0$. 
Let $V_k(c_1,\ldots,c_t)$ be the set of such vectors, where $c_i = |C_i|$. 
For two vectors $(k_1,\ldots,k_t)$ and $(k'_1,\ldots,k'_t)$ in
$V_k(c_1,\ldots,c_t)$, we say \defi{there is an $(i,j)$-move from
$(k_1,\ldots,k_t)$ to $(k'_1,\ldots,k'_t)$} if $k_\ell = k'_\ell$ for every
$\ell \neq i,j$, and~$k'_i = k_i - 1$ and~$k'_j = k_j + 1$. 
For a node~$A$ of~$F_k(D)$, let $k_j = |A \cap C_j|$ for~$j = 1,\ldots,t$. 
We say the integer vector $(k_1,\ldots,k_t)$ is the \defi{vector associated}
to~$A$. 
Note that this vector is in $V_k(c_1,\ldots,c_t)$. 

We are ready to completely characterize the condensation digraph of~$F_k(D)$ in
terms of the condensation digraph of~$D$.
See Figure~\ref{fig:charact_condensation_token_digraphs} for an example.

\begin{theorem}
\label{thm:charact_condensation_token_digraphs}
  Let~$D$ be a digraph and $C_1,\ldots,C_t$ be a topological order of its
  strongly connected components.
  The condensation digraph~$\CD(F_k(D))$ is isomorphic to the dag whose vertex
  set is $V_k(|C_1|,\ldots,|C_t|)$ and with an arc from $(k_1,\ldots,k_t)$ to
  $(k'_1,\ldots,k'_t)$ if and only if there are indices $i < j$ such that there
  was an $(i,j)$-move from $(k_1,\ldots,k_t)$ to $(k'_1,\ldots,k'_t)$ and there
  is an arc from~$C_i$ to~$C_j$ in $\CD(D)$.
\end{theorem}
\begin{proof}
  Lemma~\ref{lem:node_set_cartesian} characterized the strongly connected 
  components of~$F_k(D)$, and thus the vertex set of its condensation 
  digraph~$\CD(F_k(D))$. 
  Let $(k_1,\ldots,k_t)$ and $(k'_1,\ldots,k'_t)$ be vertices of~$\CD(F_k(D))$.
  Observe that there is an arc from $(k_1,\ldots,k_t)$ to $(k'_1,\ldots,k'_t)$
  if and only if there are two nodes~$A$ and~$B$ of~$F_k(D)$ with associated
  vectors $(k_1,\ldots,k_t)$ and $(k'_1,\ldots,k'_t)$ respectively, such that
  there is an arc in~$F_k(D)$ from~$A$ to~$B$. 

  That happens if and only if~$A$ and~$B$ differ in exactly one token, and this
  token must have moved from a vertex~$a$ to a vertex~$b$ of~$D$, through an
  arc of~$D$, with~$a$ and~$b$ being in different strongly connected components
  of~$D$.  
  Let~$C_i$ and~$C_j$ be the strongly connected components of~$D$
  containing~$a$ and~$b$ respectively.
  Then there is an arc from~$C_i$ to~$C_j$ in~$\CD(D)$ and so $i < j$. 
  Moreover, $k_\ell = k'_\ell$ for every $\ell \neq i,j$, and $k'_i = k_i - 1$
  and $k'_j = k_j + 1$. 
  Hence there was an $(i,j)$-move from $(k_1,\ldots,k_t)$ to
  $(k'_1,\ldots,k'_t)$, and the theorem holds.
\hfill$\qed$
\end{proof}

\begin{figure}[h!]
    \centering
    \resizebox{\textwidth}{!}{\includegraphics{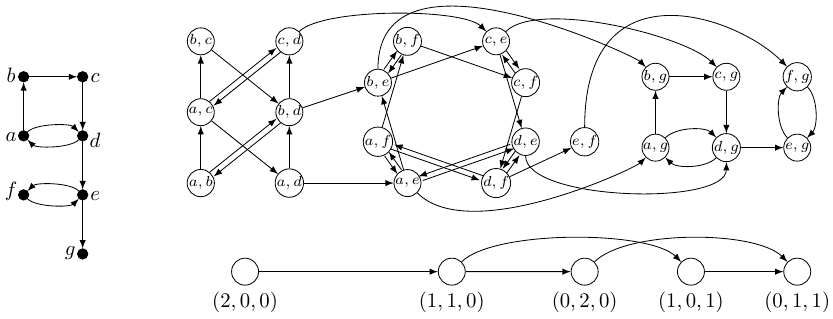}}
    \caption{A digraph~$D$ with~3 strongly connected components, the $2$-token
    digraph~$F_2(D)$ with~5 strongly connected components (for simplicity, we
    omit the bracets and place labels inside the vertices), and the digraph
    isomorphic to~$\CD(F_2(D))$ whose vertex set is $V_2(4,2,1)$.}
    \label{fig:charact_condensation_token_digraphs}
\end{figure}

In particular, we derive the following from
Theorem~\ref{thm:charact_condensation_token_digraphs}.

\begin{corollary}\label{cor:strongcomp}
  The digraph $F_k(D)$ is strongly connected if and only if $D$ is strongly
  connected. 
\end{corollary}

\begin{corollary}\label{cor:acyclic}
  The digraph $F_k(D)$ is acyclic if and only if $D$ is acyclic.
\end{corollary}

\section{Kernels}
\label{sec:kernels}

For a set~$S$ of vertices in a digraph~$D$, 
let $N^-[S] = S \cup \{u \in V(D) : (u,v) \in A(D) \mbox{ and } v \in S\}$.
A set~$K$ of vertices in~$D$ is a \defi{kernel} if~$K$ is independent and
$N^-[K] = V(D)$.
This notion was introduced by von Neumann and
Morgenstern~\cite{vonNeumannM1944}, who showed the following result.

\begin{theorem}[\cite{vonNeumannM1944}]
\label{thm:vonNeumannM}
    Every digraph with no oriented odd cycle has a unique kernel.
\end{theorem}

Theorem~\ref{thm:vonNeumannM}, together with Corollary~\ref{cor:acyclic}, also
shows that the token digraph of an acyclic digraph has a unique kernel.
Indeed, the kernel of a dag~$D$ can be obtained iteratively as follows.
Start with $D'=D$ and let~$S_1$ be the set of sinks of~$D'$; 
remove~${N^-[S_1]}$ from~$D'$ and repeat the process until~$D'$
vanishes, that is, let~$S_2$ be the set of sinks of the remaining~$D'$; 
remove~${N^-[S_2]}$ from~$D'$ and so on.
The set $K = S_1 \cup S_2 \cup \cdots$ is the unique kernel of~$D$. 

We now prove the following. 

\begin{theorem} 
  \label{thm:odd_cycles}
  If $D$ is a digraph with no oriented odd cycles, then $F_k(D)$ has no
  oriented odd cycles.
\end{theorem}  
\begin{proof}
  Because~$D$ has no oriented odd cycles, $D$ is bipartite. 
  Hence every strongly connected component~$C$ of~$D$ is also bipartite. 
  Let~$X$ and~$Y$ be a bipartition of the vertex set of~$C$. 
  For every~$i$ with $0 \leq i \leq k$, a node of~$F_i(C)$ is an $i$-token
  configuration in~$C$, and it has an even or an odd number of vertices in~$X$. 
  That induces a bipartition of~$F_i(C)$, as a move of a token goes from~$X$
  to~$Y$ or vice-versa, and always changes the parity of the configuration
  within~$X$.
  Thus, each~$F_i(C)$ is also bipartite. 

  Because each~$F_i(C)$ is bipartite, and the Cartesian product of bipartite
  digraphs is bipartite~\cite[Lemma 2.6]{Sabidussi1957}, we conclude by
  Lemma~\ref{lem:node_set_cartesian} that the strongly connected components
  of~$F_k(D)$ are bipartite, and hence~$F_k(D)$ has no oriented odd cycle. 
\hfill$\qed$
\end{proof}  

Therefore, if~$D$ has no oriented odd cycle, then~$D$ and~$F_k(D)$ have unique
kernels. 
On the other hand, if~$D$ has oriented odd cycles, then it might be the case
that~$D$ has a kernel, but~$F_k(D)$ does not, for some~$k$, and the other way
around. 
Specifically, in Figure~\ref{fig:oddcyclewithkernel}(a), we show a digraph~$D$
with a triangle that has a kernel, but for which~$F_2(D)$ does not have a
kernel, and, in Figure~\ref{fig:oddcyclewithkernel}(b), we show a digraph~$D$
that does not have a kernel, but for which~$F_2(D)$ has a kernel. 

\begin{figure}[h!]
    \centering
    \raisebox{0.2\height}{\resizebox{0.4\textwidth}{!}{\includegraphics{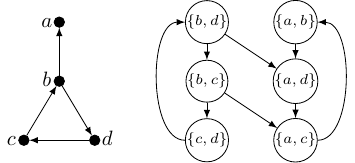}}}
    \hspace{20pt}
    \resizebox{0.45\textwidth}{!}{\includegraphics{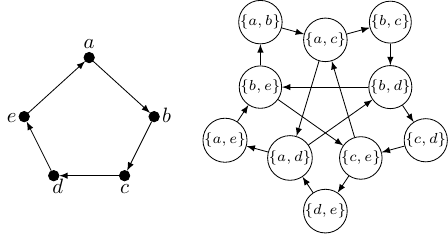}}
    \caption{Two examples of a digraph~$D$ and its $2$-token digraph~$F_2(D)$. 
    On the left, $D$ has a kernel, namely the set $\{a,c\}$, 
    while $F_2(D)$ does not. On the right, the opposite occurs: 
    the five external vertices form a kernel of~$F_2(D)$.}
    \label{fig:oddcyclewithkernel}
\end{figure}

Chvátal has shown that the problem of, given a digraph~$D$, deciding
whether~$D$ has a kernel is NP-complete~\cite{Chvatal1973}.
We adapt parts of his proof to show that the problem of, given a digraph~$D$,
deciding whether $F_2(D)$ has a kernel is also NP-complete.

Chvátal's reduction is from 3-SAT.
From a 3-SAT formula~$\phi$, he builds a digraph~$D$ such that~$D$ has a kernel
if and only if~$\phi$ is satisfied.
The idea is to use the same construction, but to add a universal sink
vertex~$u$ to~$D$.
The resulting digraph~$D'$ is such that~$\phi$ has a not-all-equal satisfying
assignment if and only if~$F_2(D')$ has a kernel.
So we would need to do the reduction from the following variant of 3-SAT, which
is also NP-complete~\cite{GareyJ1979}. 
The problem NAE-3-SAT consists of, given a 3-SAT formula, to decide whether
there is an assignment for the variables such that each clause has either one
or two true literals.
We call such an assignment a \defi{NAE assignment}. 

\begin{theorem}
  The problem of, given a digraph~$D$, deciding whether $F_2(D)$ has a kernel
  is NP-complete.
\end{theorem}
\begin{proof}
  Given a 3-SAT formula~$\phi$ on variables $x_1,\ldots,x_n$, call
  $C_1,\ldots,C_m$ the clauses in~$\phi$. 
  Consider the following digraph~$D$ built from~$\phi$. 
  For each variable~$x_j$, there is a vertex labeled~$x_j$ and a vertex
  labeled~$\bar{x}_j$ in~$D$.
  These are called \defi{literal vertices}.
  For each clause~$C_i$ there are three vertices in~$D$, each labeled by one of
  the three literals in~$C_i$. 
  These are called \defi{clause vertices}.
  The two literal vertices for variable~$x_j$ induce a digon in~$D$, 
  and these digons are called \defi{variable digons}.
  The three vertices of a clause induce an oriented triangle, and these are
  called the \defi{clause triangles}.
  Additionally, there is an arc from a vertex in a clause triangle to a literal
  vertex whenever they have the same label. 
  Finally, there is a vertex~$u$, with arcs from each literal vertex to~$u$. 
  See an example in Figure~\ref{fig:reduction}.

    \begin{figure}[h!]
        \centering
        \includegraphics{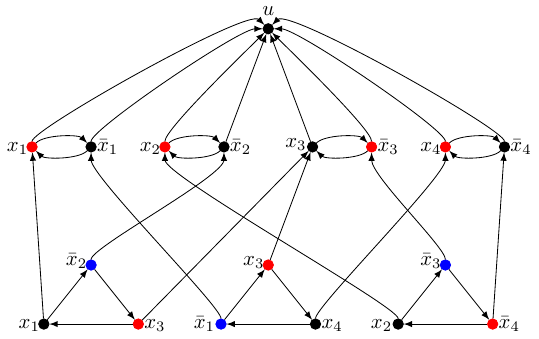}
        \caption{The digraph~$D$ for the 3-SAT formula $\phi = (x_1 \vee
        \bar{x}_2 \vee x_3) \wedge (\bar{x}_1 \vee x_3 \vee x_4) \wedge (x_2
        \vee \bar{x}_3 \vee \bar{x}_4)$. 
         The vertices in red form a kernel for the digraph $D' = D - u$.} 
        \label{fig:reduction}
    \end{figure}

  Let~$D'$ be~$D$ without vertex~$u$.
  Observe that~$D'$ is exactly the digraph from the reduction of
  Chvátal~\cite{Chvatal1973}, so we know that~$D'$ has a kernel if and only
  if~$\phi$ is satisfiable.
  Note that any kernel for~$D'$ contains a vertex in each variable digon.
  First we will prove that there is a NAE assignment for~$\phi$ 
  if and only if~$D'$ has a kernel that contains exactly one vertex 
  in each variable digon and exactly one vertex in each clause triangle.  
  We refer to such a kernel as \defi{special}.
  Second, we will prove that~$D'$ has a special kernel if and only if~$F_2(D)$
  has a kernel.  
  These two statements imply the theorem.

  Suppose there is a NAE assignment for~$\phi$.
  Let~$K$ be the set with the true literal in each variable digon, and one of
  the vertices in each clause triangle labeled by a false literal.
  When there are two vertices labeled by false literals in a clause triangle,
  choose the one that is an in-neighbor of the other.
  The red vertices in Figure~\ref{fig:reduction} show one such set~$K$
  corresponding to the assignment~${x_1 = x_2 = x_4 = T}$ and $x_3 = F$. 
  Note that~$K$ is an independent set, and every vertex in~$D'$ is either
  in~$K$ or has an arc to a vertex in~$K$. 
  Thus~$K$ is a special kernel in~$D'$. 

  Now, suppose~$K$ is a special kernel in~$D'$.
  Consider the truth assignment that makes true exactly the literals that are
  labels of literal vertices in~$K$.  
  This is well-defined because there is exactly one vertex in~$K$ in each
  variable digon.  
  We must argue that this is a NAE assignment for~$\phi$.
  In each clause triangle, there is exactly one vertex in~$K$.
  Because~$K$ is independent, the literal that labels this vertex is false in
  the assignment, and this assignment does not satisfy all three of the
  literals in each clause.
  On the other hand, for each clause triangle, one of the two vertices not
  in~$K$ is not an in-neighbor of the vertex in~$K$ in this triangle.
  As~$K$ is a kernel, this vertex has to be an in-neighbor of the literal
  vertex with the same label, which means the clause is satisfied. 
  That is, this assignment is a NAE assignment for~$\phi$. 

  Now we prove that~$D'$ has a special kernel if and only if~$F_2(D)$ has a
  kernel.  
  We start by arguing that if~$F_2(D)$ has a kernel, then $D'$ has a special kernel.
  First notice that the token graph~$F_2(D)$ has two parts.
  One of them is isomorphic to~$F_1(D')$, which in turn is isomorphic to~$D'$
  itself, and corresponds to the 2-token configurations in~$D$ that have one
  token always in~$u$.
  The second part corresponds to 2-token configurations with the two tokens
  being in~$D'$.
  Note that all the arcs between the first and the second part go from the
  second part to the first one, because~$u$ is a sink.
  Hence, any kernel~$K$ of~$F_2(D)$ induces a kernel~$K'$ in~$D'$, 
  namely, take $K' = \{x \in V(D'): \{u,x\} \in K\}$.  
  Let us argue that~$K'$ is a special kernel in~$D'$. 
  To ease the exposition, let us refer to a 2-token configuration in~$D$ as one
  of the following types, depending on where the two tokens are: a $uL$, $uC$,
  $LL$, $LC$, or a $CC$ configuration. 

  Suppose, by means of a contradiction, that~$K'$ is not special.  
  This means that there is a triangle clause~$\triangle$ with no vertex
  in~$K'$. 
  Because~$K'$ is a kernel in~$D'$, from each vertex~$y$ in~$\triangle$, the
  arc from~$y$ to~$L$ goes to a vertex in~$K'$.  
  Let~$x$ be an arbitrary variable, and consider the 2-token configurations
  with one token in the variable digon associated with~$x$ and the other
  in~$\triangle$.  
  Call~$Z$ this set of~$LC$ configurations.  
  Let us argue that no neighbor of configuration in~$Z$ outside~$Z$ are in~$K$. 
  There are arcs from~$LC$ configurations to~$uC$, $LL$, and~$LC$
  configurations. 
  The~$uC$ configurations that receive arcs from configurations in~$Z$ are not
  in~$K$, because they consist of~$u$ and a vertex in~$C \setminus K'$.
  The~$LL$ configurations that receive arcs from configurations in~$Z$ are also
  not in~$K$, because at least one of the two vertices in~$L$ is in~$K'$. 
  Moreover, if there is an arc from a configuration in~$Z$ to an~$LC$
  configuration, the latter is also in~$Z$. 
  Hence, $K$ must contain a kernel of the subdigraph of~$F_2(D)$ induced
  by~$Z$.
  See Figure~\ref{fig:zica} for an example of such subdigraph.
  However, one can check that this digraph has no kernel. 

\begin{figure}[h!]
  \centering
  \resizebox{0.28\textwidth}{!}{\includegraphics{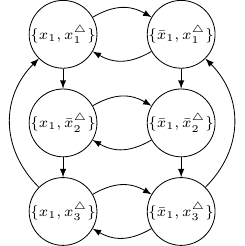}}
  \caption{The subdigraph of~$F_2(D)$ for~$D$ from Figure~\ref{fig:reduction},
  with~$\triangle$ being the triangle for clause $x_1 \vee \bar{x}_2 \vee x_3$
  and $x = x_1$.
  We name $x^\triangle_1,\bar{x}^\triangle_2,x^\triangle_3$ the vertices in
  $\triangle$ whose labels are $x_1,\bar{x}_2,x_3$ respectively.}
  \label{fig:zica}
\end{figure}

  It remains to show that if~$D'$ has a special kernel, then~$F_2(D)$ has a kernel.
  Let~$K'$ be a special kernel in~$D'$.
  There is exactly one vertex from~$K'$, say~$y$, in each clause
  triangle~$\triangle$.
  The vertex not in~$K'$ in~$\triangle$ that is an in-neighbor of~$y$ is called
  the \emph{dominating vertex} in~$\triangle$ while the other vertex
  in~$\triangle$ not in~$K'$ is called \emph{undominating vertex}. 
  Note that the undominating vertex is an in-neighbor of a literal vertex in~$K$.
  Let us denote by $C \subseteq V(D)$ the set of clause vertices and by $L
  \subseteq V(D)$ the set of literal vertices.

  Let us now describe a kernel in~$F_2(D)$ from~$K'$.
  Let $S_1 = \{\{u,\textcolor{red}{v}\} \colon \textcolor{red}{v} \in K'\}$. 
  This is an independent set, because~$K'$ is an independent set in $D'$.
  Let $S_2 = \{\{v,v'\} \colon v,v' \in L$ and $v,v' \notin K'\}$.
  Because~$K'$ is special, there is a vertex in~$K'$ in each variable digon. 
  Thus, as there is no arc between variable digons, the 2-token configurations
  in~$S_2$ form an independent set. 
  In fact, $S_1 \cup S_2$ is also an independent set, because the tokens in
  configurations of~$S_2$ are in vertices not in~$K'$, while the tokens in
  literal vertices in configurations of~$S_1$ are in vertices from~$K'$. 
  Let $S_3 = \{\{v,y\} \colon$ either $\textcolor{red}{v} \in K' \cap L$ and
  $\textcolor{blue}{y} \in C$ is a dominating vertex, or $v \in L \setminus K'$
  and $y \in C$ is an undominating vertex$\}$.
  Clearly, $S_3$ is an independent set.
  Also, because a token in~$C$ cannot move to~$u$, and a token in an
  undominating vertex can only move to a vertex in~$L$ that is in~$K'$, 
  we can see that $S_1 \cup S_2 \cup S_3$ is also an independent set.
  At last, let $S_4 = \{\{y,z\} \colon y,z \in C$ and either
  $\textcolor{red}{y},\textcolor{red}{z} \in K'$, or
  $\textcolor{blue}{y},\textcolor{blue}{z}$ are both dominating vertices not
  neighboring a vertex in $L \cap K'$, or $y,z$ are both undominating vertices,
  or both~$y$ and~$z$ are in the same clause triangle with $\textcolor{red}{y}
  \in K'$ and~$\textcolor{blue}{z}$ being a dominating vertex$\}$.
  Clearly, $S_4$ is independent. 
  Also, we can see that $K = S_1 \cup S_2 \cup S_3 \cup S_4$ is also an
  independent set, because a token in an undominating vertex can only go to a
  vertex in~$L$ that is in~$K'$. 

  Finally, let us argue that~$K$ is a kernel of~$F_2(D)$. 
  Consider an arbitrary node~$S$ of~$F_2(D)$.
  If~${S \in K}$, there is nothing to prove, so we may assume $S \notin K$.
  If $u \in S$, then the second vertex in~$S$ is not in~$K'$, otherwise~$S$
  would be in~$S_1$, so there is an arc from~$S$ to a configuration in~$S_1$,
  because~$K'$ is a kernel in~$D'$. 
  If $u \notin S$, then there are three cases.
  In the first case, $S$ contains only literal vertices.
  Then, because $S \not\in S_2$, there is a vertex from~$K'$ in~$S$, and hence
  there is an arc from~$S$ to a configuration in~$S_1$.
  In the second case, $S = \{v,y\}$ contains a literal vertex~$v$ and a clause
  vertex~$y$.
  If $y \in K'$, then there is an arc from~$S$ to the configuration $\{u,y\}$
  in~$S_1$.
  Thus either~$\textcolor{blue}{y}$ is a dominating vertex and~$v \notin K'$,
  or~$y$ is an undominating vertex and~$\textcolor{red}{v} \in K'$
  (otherwise~$S$ would be in~$S_3$), and either way there is an arc from~$S$ to
  a configuration in~$S_3$ (obtained by moving the token in~$L$ to the other
  literal vertex in the same digon).
  In the last case, $S = \{y,z\}$ contains only clause vertices.
  If $\{y, z\} \cap K' = \emptyset$, then one of them is dominating and the
  other (dominating or undominating) has a neighbor in~$X \cap K'$, which means
  there is an arc from~$S$ to a configuration in~$S_3$. 
  Otherwise we may assume~$\textcolor{red}{y} \in K'$, and
  either~$\textcolor{blue}{z}$ is the dominating vertex in another clause
  triangle, or~$z$ is undominating, and there is always an arc from~$S$
  to~$S_4$.
  Indeed, in the former case, we can move the token in~$\textcolor{blue}{z}$ to
  the vertex in~$K'$ in the same triangle and, in the latter case, if
  both~$\textcolor{red}{y}$ and~$z$ are in the same triangle, then we can move
  the token in~$z$ to the dominating vertex in the same triangle, and
  if~$\textcolor{red}{y}$ and~$z$ are in different clause triangles, we can
  move the token from~$\textcolor{red}{y}$ to the undominating vertex in the
  same triangle. 
\hfill$\qed$
\end{proof}


\section{Unilateral digraphs and their token digraphs}
\label{sec:uni}

A digraph~$D$ is \defi{unilateral} if, for every pair of vertices~$x$ and~$y$,
there is an $xy$-path or a $yx$-path (or both).
Note that a digraph can be weakly connected without being unilateral (take the
antipath for instance).
We use the next theorem to characterize when~$F_k(D)$ is unilateral.

\begin{theorem}[Theorem 7.2 in~\cite{Foulds1991}]
\label{thm:Hamilton}
    A digraph~$D$ is unilateral if and only if the condensation
    digraph~$\CD(D)$ has a Hamiltonian path. 
\end{theorem}

Because~$\CD(D)$ is a dag for every digraph~$D$, the Hamiltonian path of a
unilateral digraph given by Theorem~\ref{thm:Hamilton} is unique. 

Obviously, as~$F_1(D)$ is isomorphic to~$D$, if~$D$ is unilateral, then so
is~$F_1(D)$. 
Moreover, $F_1(D)$ is isomorphic to~$F_{n-1}(D)$, where~$n$ is the number of
vertices in~$D$. 
Hence the next theorem addresses the remaining cases, that is, $2 \leq k \leq n
- 2$.

\begin{theorem}
    Let~$D$ be an $n$-vertex digraph with $C_1,\ldots,C_t$ being its strongly
    connected components, and let~$k$ be such that $2 \leq k \leq n - 2$.
    Then~$F_k(D)$ is unilateral if and only if~$D$ is unilateral and either $t
    \leq 2$ or $t=3$ with $|C_2| = 1$.
\end{theorem}
\begin{proof}
    Let us first show that if~$D$ is unilateral and either $t \leq 2$ or $t=3$
    with $|C_2| = 1$, then $F_k(D)$ is unilateral.
    Since~$D$ is unilateral, let $C_1,\ldots,C_t$ be the strongly connected
    components of~$D$ in the order given by the Hamiltonian path of
    Theorem~\ref{thm:Hamilton}.

    We start by considering the case $t \leq 2$.
    If $t = 1$, then~$D$ is strongly connected and so it is~$F_k(D)$ by
    Corollary~\ref{cor:strongcomp}, and hence~$F_k(D)$ is unilateral. 
    For $t = 2$, according to
    Theorem~\ref{thm:charact_condensation_token_digraphs}, the vertices
    of~$\CD(F_k(D))$ are pairs $(k_1,k_2)$, for $0 \leq k_i \leq |C_i|$
    and~${k_1 + k_2 = k}$. 
    Let $k_1 = \min\{k,|C_1|\}$ and $k_2 = \min\{k,|C_2|\}$. 
    The following is a Hamiltonian path in~$\CD(F_k(D))$:
    \begin{equation*}
        ((k_1,k-k_1), (k_1-1,k-k_1+1), (k_1-2,k-k_1+2), \ldots, (k-k_2,k_2)) \, .
    \end{equation*}
    Therefore, by Theorem~\ref{thm:Hamilton}, $F_k(D)$ is unilateral.
    
    Now, suppose that~$D$ is unilateral, $t=3$ and $|C_2| = 1$. 
    If $k \leq |C_1|$, then the following is a Hamiltonian path
    in~$\CD(F_k(D))$:
    \begin{equation*}
        ((k,0,0), (k-1,1,0), (k-1,0,1), (k-2,1,1), (k-2,0,2), (k-3,1,2), (k-3,0,3), \ldots, v) \, ,
    \end{equation*}
    where the last vertex~$v$ is either $(0,0,k)$ if $k \leq |C_3|$, or
    $(k-|C_3|-1,1,|C_3|)$ otherwise.
    If $k > |C_1|$, then the following is a Hamiltonian path in~$\CD(F_k(D))$:
    \begin{equation*}
        ((|C_1|,1,k-|C_1|-1), (|C_1|,0,k-|C_1|), (|C_1|-1,1,k-|C_1|), (|C_1|-1,0,k-|C_1|+1), \ldots, v) \, ,
    \end{equation*}
    where again~$v$ is either $(0,0,k)$ if $k \leq |C_3|$, or
    $(k-|C_3|-1,1,|C_3|)$ otherwise.
    
    For the other direction, we prove the contrapositive, that is, we show that
    if~$D$ is not unilateral or $t \geq 4$ or $t=3$ with $|C_2| \geq 2$, then
    $F_k(D)$ is not unilateral. 

    Suppose that $t \geq 4$ or $t=3$ with $|C_2| \geq 2$.
    Let $C' = C_2 \cup \cdots \cup C_{t-1}$, so $|C'| \geq 2$.  
    Take~$A$ to be an arbitrary~$k$-token configuration of~$D$ such that $|A
    \cap C_1| \geq 1$, $|C' \setminus A| \geq 2$, and $|A \cap C_t| \geq 1$. 
    Such a set exists because $k \geq 2$, which assures we can select~$A$ with
    $|A \cap C_1| \geq 1$ and $|A \cap C_t| \geq 1$, and because $|C'| \geq 2$
    and $k \leq n-2 = |C_1| + |C'| + |C_t| - 2$, which assures we can also
    choose~$A$ that does not contain at least two vertices~$w$ and~$z$ in~$C'$.
    Let~$x$ and~$y$ be vertices in $A \cap C_1$ and~${A \cap C_t}$,
    respectively. 
    Let~$B$ be the $k$-token configuration $A \setminus\{x,y\} \cup \{w,z\}$. 
    Then, there is no path from~$A$ to~$B$ in~$F_k(D)$ because there is no way
    to move tokens from~$A$ to a configuration with less vertices in~$C_t$,
    such as~$B$. 
    Also, there is no path from~$B$ to~$A$ because there is no way to move
    tokens from~$B$ to a configuration with more vertices in~$C_1$, such
    as~$A$. 
    Therefore, $F_k(D)$ is not unilateral.
    
    Suppose now that~$D$ is not unilateral.
    Then clearly $t \geq 3$.
    By the previous case, we may assume that $t = 3$ and~$|C_2| = 1$. 
    As we observed in Section~\ref{sec:strong_connectivity}, $D$ is weakly
    connected if and only if~$F_k(D)$ is weakly connected, so we may assume
    that~$D$ is weakly connected. 
    Let~$y$ be the unique vertex in~$C_2$. 
    Since~$D$ is not unilateral, $\CD(D)$ has exactly two arcs, which
    are~${\{(C_1,C_3),(C_2,C_3)\}}$ or~${\{(C_1,C_2),(C_1,C_3)\}}$.
    Note that if $\CD(D)$ has the arcs $\{(C_1,C_3),(C_2,C_3)\}$, then
    $\CD(\larrow{D})$ has the arcs $\{(C_1,C_2),(C_1,C_3)\}$, and vice-versa. 
    So, given that~$D$ is unilateral if and only if~$\larrow{D}$ is unilateral,
    we may assume the former case. 
    Now, take a configuration~$A$ with $y \notin A$ and~$|A \cap C_1| \geq 1$.
    Let $x \in A \cap C_1$, and let $B := (A \setminus \{x\}) \cup \{y\}$.
    Observe that there is no path from~$A$ to~$B$ because no token at vertices
    in~$A$ can be moved to the vertex $y \in C_2$. 
    Further, there is no path from~$B$ to~$A$ because no token at vertices in
    $B \cap (C_2 \cup C_3)$ can be moved to a vertex in~$C_1$.
    Hence, $F_k(D)$ is not unilateral.
\hfill$\qed$
\end{proof}

\section{Girth and circumference}
\label{sec:girth}

The \defi{oriented girth} of a digraph~$D$ is the length of the shortest
oriented cycle in~$D$, and it is denoted by~$g(D)$.
The \defi{oriented circumference} of~$D$ is the length of the longest oriented
cycle in~$D$, and it is denoted by~$c(D)$. 

\begin{theorem}
\label{thm:girth_circumf}
  For every digraph~$D$ and for each $k \in \{1,2 \ldots, |D|\}$, $g(D) =
  g(F_k(D))$ and $c(D) \leq c(F_k(D))$. 
\end{theorem}
\begin{proof}
  Let~$C$ be an oriented cycle in~$D$ of length~$t$. 
  Let~$A$ be a $k$-token configuration in~$D$ with at least one and at
  most~$t-1$ tokens in~$C$. 
  Such a configuration exists because $1 \leq k \leq |V(D)|-1$. 
  We can move the tokens in~$C$ around to obtain an oriented cycle in~$F_k(D)$
  of length exactly the length of~$C$, that is, $t$.
  This implies that $g(D) \geq g(F_k(D))$ and $c(D) \leq c(F_k(D))$.

  Now, let~$C$ be an oriented cycle in~$F_k(D)$ of length~$t$. 
  Let us argue that there is a corresponding oriented cycle~$C'$ in~$D$ of
  length at most the length of~$C$. 
  Let $C = (A_0,\ldots,A_t)$, where $A_0 = A_t$.  
  Let $U = \bigcup_i A_i$ and let~$D'$ be the subdigraph of~$D$ with vertex
  set~$U$ and with an arc from~$u$ to~$v$ if and only if $A_{i-1} \setminus A_i
  =\{u\}$ and $A_i \setminus A_{i-1} =\{v\}$ for some $1 \leq i \leq t$.
  Note that $|A(D')| = t$.

  While traversing~$C$, each of the~$k$ tokens traverses an oriented path from
  a vertex in~$A_0$ back to a vertex in~$A_0$ (the same vertex or another one). 
  Hence every vertex of~$D'$ is either isolated or has outdegree at least one,
  and there is at least one vertex with outdegree at least one because $t \geq 1$.
  Throw away the isolated vertices.
  What remains is a digraph all of whose vertices have outdegree at least one.
  Every such digraph has a cycle. 
  Clearly this cycle has length at most~$t$, as~$D'$ has exactly~$t$ arcs.
  For $t = g(F_k(D))$, this means that $g(D) \leq g(D') \leq t = g(F_k(D))$,
  which allows us to conclude that $g(D) = g(F_k(D))$. 
  \hfill$\qed$
\end{proof}

Now we consider the oriented circumference of token digraphs.
Since~$F_1(D)$ is isomorphic to~$D$ and~$F_{n-1}(D)$ is isomorphic
to~$\larrow{D}$, we have that $c(F_1(D)) = c(F_{n-1}(D)) = c(D)$. 
On the other hand, when $2 \leq k \leq n-2$, there exist many digraphs~$D$ such
that $c(D) < c(F_k(D))$. 
Indeed, $F_k(D)$ is much larger than~$D$ and, for instance, $F_k(\dvec{K_n})$
is a Hamiltonian digraph, and so $c(F_k(\dvec{K_n})) > c(\dvec{K_n})$ as long
as $n \geq 4$ and $2 \leq k \leq n-2$.

Also, there exist non-Hamiltonian graphs whose $k$-token graphs are Hamiltonian
(see~\cite{AdameRT2021}), and we can consider the digraph~$D$ obtained
from such graphs by replacing each edge by the two possible arcs.
It is straightforward to see that~$F_k(D)$ is Hamiltonian, so the same holds
for these digraphs.

\begin{theorem}
  Let~$D$ be a digraph on~$n$ vertices with $c(D) \geq 5$ and $2 \leq k \leq
  n-3$.
  Let $r = 2$ if $k = 2$ and $r = \min\{\max\{k, n-k\}, c(D)-3\}$ otherwise. 
  Then $c(F_k(D)) \ge r \, c(D)$.
\end{theorem}
\begin{proof}
  As $c(D) = c(\larrow{D})$ and $F_k(D)$ is isomorphic to the token graph
  $F_{n-k}(\larrow{D})$, we can assume that either $k=2$ or~$k \geq n-k$.
  Let $c = c(D)$ and~$C$ be a cycle of length~$c$ in~$D$. 
  Assume without loss of generality that $V(C) = \{0,1,\dots,c-1\}$. 
  In what follows, sums and subtractions are taken $\bmod\;r$. 

  First suppose that $r = k \le c-3$. 
  For each $i \in V(C)$, let $X_i = \{i,i+1,\dots,i+k\}$.
  Clearly $|X_i \cap X_{i+1}| = k$.   
  Now, note that $|V(C) \setminus X_i| = c - (k+1) \geq 2$, which means that
  $|X_i \cap X_j| < k$ if $|j-i| \neq 1$, because~$i$ and~$i+1$ are not
  in~$X_j$ for $j = i+2,\ldots,i+k-1$ (as the last element in~$X_j$ in this
  case would be at most $i+k-1+k = i-1$), $i+1$ and~$i+k-1$ are not
  in~$X_{i+k}$ (as the last element in $X_{i+k}$ is~$i$), and~$i+k-1$ and~$i+k$
  are not in~$X_j$ for $j = i+k+1,\ldots,i-2$ (as the last element in~$X_j$ in
  this case would be at most $i+k-2$).  

  Let~$P_i$ be the directed path in~$F_k(D)$ from $X_{i-1} \cap X_i =
  \{i,i+1,\dots,i+k-1\}$ to $X_i \cap X_{i+1} = \{i+1,i+1,\dots,i+k\}$,
  obtained from moving one by one the token from~$j$ to~$j+1$, for $j =
  i+k-1,\ldots,i$, where the sum is taken $\bmod\; r$. 
  It is readily seen that~$P_i$ is a directed path in~$F_k(D)$ of length~$k$,
  and that its vertices correspond to $k$-token configurations contained
  in~$X_i$. 
  Because of this, as $|X_i \cap X_j| < k$ if $|j-i| \neq 1$, paths~$P_i$
  and~$P_j$ do not intersect if $|j-i| \neq 1$.
  Moreover, $P_i$ starts at the end of~$P_{i-1}$, hence we can concatenate
  $P_0,P_1,\ldots,P_{c-1}$ to obtain a cycle of length~$kc$ in~$F_k(D)$, which
  implies that $c(F_k(D)) \geq kc = rc$. 
 
  Now suppose that $r = c-3 < k$.
  As $k \leq n-3$, there are $k-r$ vertices outside~$C$. 
  Consider the subdigraph~$H$ of $F_k(D)$ generated by moving~$r$ tokens
  on~$V(C)$, whereas the remaining $k-r$ tokens are fixed outside~$C$. 
  Thus, $H \simeq F_r(C)$ and, then, by the first part of the proof, we deduce
  that there is a cycle in~$H$ of order~$rc$.
  This completes the proof.  
  \hfill$\qed$
\end{proof}

Cordero-Michel and Galeana-Sánchez~\cite{Cordero2019} proved the following.

\begin{theorem}[\cite{Cordero2019}]
  \label{thm:dichi_girth}
  Let~$D$ be a digraph with at least one cycle.
  Then $\dichi(D) \leq \left\lceil\dfrac{c(D)-1}{g(D)-1}\right\rceil + 1$.
\end{theorem}

Using Theorem~\ref{thm:dichi}, we can conclude the following.

\begin{corollary}
  Let~$D$ be a digraph with at least one cycle.
  Then $\dichi(F_k(D)) \leq \left\lceil\dfrac{c(D)-1}{g(D)-1}\right\rceil + 1$.
\end{corollary}

\section{Eulerian and Hamiltonian digraphs}
\label{sec:eulerian}

A digraph~$D$ is \defi{Eulerian} if $d^+(v) = d^-(v)$ for all $v \in V(D)$.

\begin{theorem}
  For every~$k$, the digraph~$D$ is Eulerian if and only if~$F_k(D)$ is
  Eulerian.
\end{theorem}
\begin{proof}
  Let~$A$ be a $k$-token configuration and denote by $e(A)$ the number of arcs
  of~$D$ whose both ends are in~$A$. 
  Then note that $d^+(A) = \sum_{v \in A} d^+(v) - e(A)$ and $d^-(A) = \sum_{v
  \in A} d^-(v) - e(A)$. 

  Suppose first that~$D$ is Eulerian. 
  Since $d^+(v) = d^-(v)$ for every $v \in V(D)$, we conclude directly from the
  observation above that $d^+(A) = d^-(A)$ for every~$A$, and so $F_k(D)$ is
  Eulerian. 

  Suppose now that~$F_k(D)$ is Eulerian.
  For any two distinct vertices~$u$ and~$v$ of~$D$, let~$A$ be a $k$-token
  configuration containing~$u$ but not~$v$, and let $A' = A \setminus \{u\}
  \cup \{v\}$.
  Note that $A \setminus \{u\} = A' \setminus \{v\}$ and let $S^+ = \sum_{x \in
  A \setminus \{u\}} d^+(x)$ and $S^- = \sum_{x \in A \setminus \{u\}}
  d^-(x)$.
  Then
  \begin{equation*}
    d^+(u) = d^+(A) - S^+ + e(A) 
           = d^-(A) - S^+ + e(A) 
           = S^- + d^-(u) - e(A) - S^+ + e(A) 
           = d^-(u) + S^- - S^+.
  \end{equation*}
  Analogously, using~$A'$ instead of~$A$, we can derive that $d^+(v) =
  d^-(v)+S^- - S^+$.
  Therefore $d^+(u) - d^-(u) = S^- - S^+ = d^+(v) - d^-(v)$ and, from this, 
  we conclude that ${\sum_{x \in V(D)} d^+(x) - \sum_{x \in V(D)} d^-(x) = |V(D)|(S^- - S^+)}$.
  Because $\sum_{x \in V(D)} d^+(x) = \sum_{x \in V(D)} d^-(x)$, it must be the
  case that $S^- - S^+ = 0$, and hence $d^+(x) = d^-(x)$ for every~$x \in V(D)$.
\hfill$\qed$
\end{proof}

For (non-directed) graphs on~$n$ vertices, we know that $F_2(C_n)$ is
Hamiltonian if and only if $n=3$ or $n = 5$. 
The same statement holds for directed cycles and the directed token digraph. 
On the other hand, the digraph~$D$ shown in Figure~\ref{fig:nonHam} is not 
Hamiltonian but its token digraph $F_2(D)$ is Hamiltonian. 

\begin{figure}[h!]
  \centering
  \resizebox{0.6\textwidth}{!}{\includegraphics{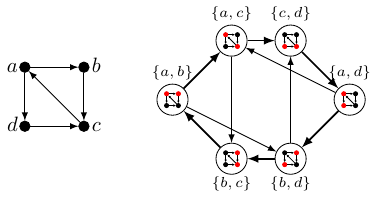}}
  \caption{A non-Hamiltonian digraph with a Hamiltonian token digraph.}
  \label{fig:nonHam}
\end{figure}

\section{Bidirected cliques}
\label{sec:clique}

We call the digraph $\lrarrow{K_n}$ the \defi{complete digraph} on~$n$
vertices. 
A~\defi{bidirected clique} in a digraph~$D$ is a complete subdigraph of~$D$.
The \defi{bidirected clique number} of a digraph~$D$, denoted
by~$\diclique(D)$, is the size of the largest bidirected clique of~$D$.
We use the next known result on the clique number of undirected token graphs to
characterize the bidirected clique number of~$F_k(D)$.

\begin{theorem}[Theorem 5 in \cite{FabilaMonroyFHHUW2012}]
  \label{thm:cliques}
  For any graph~$G$ of order~$n$ and $1 \leq k \leq n-1$, it holds that
  $\clique(F_k(G)) = \min\{\clique(G), {\max\{n-k+1, k+1\}}\}$.
\end{theorem}

The \defi{clean graph} of a digraph~$D$, denoted by $D^*$, is the graph
whose vertex set is~$V(D)$ and there is an edge between~$u$ and~$v$ if and only
if $(u,v),(v,u) \in A(D)$.
The next result states that constructing the token graph of~$D$ and then
cleaning it is the same as cleaning~$D$ first and then constructing the token
graph of~$D^*$.
Observe that $F_k(D)^*$ and $F_k(D^*)$ are both (undirected) graphs.

\begin{fact}
  \label{fact:cleaning}
  For any digraph~$D$, $F_k(D)^* = F_k(D^*)$.
\end{fact}
\begin{proof}
  Let~$X$ and~$Y$ be two distinct nodes in~$F_k(D)^*$.
  We will prove that~$XY$ is an edge in $F_k(D)^*$ if and only if~$XY$ is an
  edge in $F_k(D^*)$.
  We need to consider only nodes~$X$ and~$Y$ whose symmetric difference is a
  pair of vertices in~$D$.
  So, assume that $(X \setminus Y) \cup (Y \setminus X) = \{x,y\}$, for some
  $x,y \in V(D)$.
  By the definition of clean graph, $XY$ is an edge in $F_k(D)^*$ if and only
  if $(X,Y)$ and $(Y,X)$ are arcs in~$F_k(D)$.
  Now, by the definition of token graph, $(X,Y),(Y,X) \in A(F_k(D))$ if and
  only if $(x,y),(y,x) \in A(D)$. 
  Using the definition of clean graph, $(x,y), (y,x) \in A(D)$ if and only
  if~${xy \in E(D^*)}$.
  Finally, ${xy \in E(D^*)}$ if and only if~$XY$ is an edge in $F_k(D^*)$.
\hfill$\qed$
\end{proof}

\begin{theorem}
  For any digraph~$D$ of order~$n$ and $1 \leq k \leq n-1$, it holds that
  $\diclique(F_k(D)) = \min\{\diclique(D),\max\{n-k+1, k+1\}\}$.
\end{theorem}
\begin{proof}
  First, note that $\diclique(D) = \clique(D^*)$.
  Using this and Fact~\ref{fact:cleaning}, we have that
  \begin{equation*}
    \diclique(F_k(D)) = \clique(F_k(D)^*) = \clique(F_k(D^*)) \, .
  \end{equation*}
  By Theorem~\ref{thm:cliques},
  $\clique(F_k(D^*)) = \min\{ \clique(D^*), \max\{n-k+1, k+1\}\}$.
  Therefore, 
  \begin{equation*}
    \diclique(F_k(D)) = \min\{ \clique(D^*), \max\{n-k+1, k+1\}\} 
                      = \min\{ \diclique(D), \max\{n-k+1, k+1\}\} \, .
  \end{equation*}
\hfill$\qed$
\end{proof}

\section{Acyclic partitions}
\label{sec:acyclic_partition}

An \defi{acyclic $r$-partition} of a digraph~$D$ is a partition of its vertex
set into~$r$ sets such that each one induces an acyclic subdigraph of~$D$.
We can see such partition as a (non-proper) coloring $c \colon V(D) \to
\{1,\ldots,r\}$ of the vertices and each set as a color class.
The~\defi{dichromatic number} of~$D$, denoted by $\dichi(D)$, is the smallest
integer~$r$ such that~$D$ has an acyclic $r$-partition.
These notions were introduced by Neumann-Lara~\cite{Neumann1982} as a
generalization of proper coloring and chromatic number in undirected graphs.

By Corollary~\ref{cor:acyclic}, we have~$\dichi(D) = 1$ if and only
if~$\dichi(F_k(D))=1$.
Neumann-Lara~\cite{Neumann1982} proved that if~$D$ has no oriented odd cycles,
then~$\dichi(D) \leq 2$.
Thus, by Theorem~\ref{thm:odd_cycles}, we can conclude that if~$D$ has no
oriented odd cycles, then~$\dichi(F_k(D)) \leq 2$.
In fact, we can show the following result.

\begin{theorem}
\label{thm:dichi}
    For any digraph~$D$, $\dichi(F_k(D)) \leq \dichi(D)$.
\end{theorem}
\begin{proof}
    Let~$c \colon V(D) \to \{1,\ldots,r\}$ be an optimal acyclic partition
    of~$D$, with $r = \dichi(D)$, and let $H_1,\dots,H_r$ be the subdigraphs
    of~$D$ induced by each of the color classes.
    As~$c$ is an acyclic partition, each~$H_i$ is acyclic. 
    For each node $A \in F_k(D)$, let 
    \begin{equation*}
        c'(A) = \sum\limits_{a\in A} c(a) \mod r \, .
    \end{equation*}
    We aim to show that~$c'$ is an acyclic partition of~$F_k(D)$.
    For this purpose, let us define an auxiliary $r$-vector for the nodes
    of~$F_k(D)$. 
    For a node $A \in F_k(D)$, let $\tau(A) = (t_1,\dots,t_r)$, with $t_i = |A
    \cap V(H_i)|$ for $i \in [r]$.

    Consider an arc $(A,B)$ of $F_k(D)$, and let $(a,b) \in A(D)$ be the
    corresponding arc such that~$(A,B)$ is generated by sliding one token along
    $(a,b)$.
    Observe that $(a,b)$ belongs to~$H_i$, for some $i \in [r]$, if and only if
    $\tau(A) = \tau(B)$. 
    Thus, 
    \begin{equation*}
        c'(A) = c'(B) \iff \tau(A) = \tau(B) \, .
    \end{equation*}
    This observation is generalized as follows. 
    If $J \subseteq F_k(D)$ is a weakly connected subdigraph contained in a
    same color class of~$F_k(D)$, then $\tau(A) = \tau(B)$ for any two nodes
    $A,B \in J$.
    This fact implies that~$J$ is generated by moving $k_1,\dots,k_r$ tokens on
    $H_1,\dots,H_r$, respectively, where $0 \le k_i \le |V(H_i)|$ and $k_1 +
    \dots + k_r = k$. 
    In particular, we have that the tokens moving on a class~$H_i$ cannot slide
    to any other class~$H_j$, implying that~$J$ contains no oriented cycle.

    On the other hand, note that a color class~$\mathcal{H}$ of~$F_k(D)$ is a
    disjoint union of maximal weakly connected subdigraphs of~$F_k(D)$ having
    the same color in~$c'$, and given that each of these subdigraphs contains
    no oriented cycle, we conclude that~$\mathcal{H}$ contains no oriented
    cycle. 
    Therefore, $c'$ is an acyclic partition of~$F_k(D)$, and by the definition
    of~$c'$, we have $\dichi(F_k(D)) \le r = \dichi(D)$, as we wanted. 
\hfill$\qed$
\end{proof}

The~\defi{chromatic number} of a graph~$G$, denoted by $\chi(D)$, is the smallest
integer~$k$ such that~$G$ has a proper $k$-coloring.
In the undirected case, it is known that $\chi(F_k(K_n)) < n = \chi(K_n)$ for
some values of~$k$ and~$n$. 
Specifically, $\chi(F_2(K_n)) = n-1$ for any even~$n$. 
Note that $\dichi(\lrarrow{G}) = \chi(G)$. 
Hence, by Property~\ref{property3}, $\dichi(F_2(\lrarrow{K_n})) =
\dichi(\lrarrow{F_2(K_n)}) = \chi(F_2(K_n)) = n-1 < \chi(K_n) =
\dichi(\lrarrow{K_n})$ for every even~$n$.
It would be nice to understand the behavior of~$\chi(F_k(G))$
and~$\dichi(F_k(D))$ as~$k$ varies. 
For the undirected case, we know we may restrict attention to $1 \leq k \leq
n/2$. 
Similarly, because $\dichi(D) = \dichi(\larrow{D})$ and $\dichi(F_k(D)) =
\dichi(F_{n-k}(\larrow{D}))$, we may also restrict attention to $1 \leq k \leq
n/2$ in the directed case. 

Neumann-Lara~\cite{Neumann1982} proved that 
\begin{equation*}
    \dichi(D) = \max\{\dichi(C) \colon C \text{ is a strongly connected component of } D\}\;.
\end{equation*}
Let~$C$ be a strongly connected component of~$D$ such that
$\dichi(D)=\dichi(C)$. 
If $|V(D)| \geq |V(C)|+k-1$, then~${\dichi(F_k(D)) \geq \dichi(D)}$. 
Indeed, we can fix~$k-1$ tokens in~$k-1$ nodes outside~$C$ and leave one in~$C$
to derive that $F_1(C) \subseteq F_k(D)$.
Consequently, $\dichi(F_k(D)) \geq \dichi(F_1(C)) = \dichi(C) = \dichi(D)$, and
therefore $\dichi(F_k(D)) = \dichi(D)$ by Theorem~\ref{thm:dichi}.
In particular, $D$ must have at least two strongly connected components for
$|V(D)| \geq |V(C)|+k-1$ to hold. 
If~$D$ is strongly connected, then it might happen that $\dichi(F_k(D)) <
\dichi(D)$ for some~$k$, as we previously pointed out.
 
One might ask whether $\dichi(F_2(D)) < \dichi(D)$ for every strongly connected
graph~$D$, but this is not true.
For instance, $\dichi(F_k(\rarrow{C_n})) = \dichi(\rarrow{C_n}) = 2$ for the
oriented cycle $\rarrow{C_n}$.
(In particular, see $\rarrow{C_5}$ and its 2-token digraph in
Figure~\ref{fig:oddcyclewithkernel}.)
It would also be nice to find out the exact conditions that assure that
${\dichi(F_k(D)) = \dichi(D)}$ holds, even considering only strongly connected
digraphs~$D$ and~$k=2$.

While investigating this topic, we considered the same question for an
undirected graph~$G$ and~$\chi$, as it coincides with the question for the
digraph~$\lrarrow{G}$ and~$\dichi$.
We conjecture the following.

\begin{conjecture}\label{conj:nossa}
  $\chi(F_2(G)) < \chi(G)$ if and only if $G = K_n$ for~$n$ even. 
\end{conjecture}

Let~$G$ be a graph of order~$n$.
A proper coloring of~$F_2(G)$ can be viewed as an edge-coloring of~$K_n$ (not
necessarily proper) such that, if the edges~$uv$ and~$vw$ of~$K_n$ are assigned
the same color, then $uw \notin E(G)$. 
(If $uw \in E(G)$, then $\{u,v\}$ and~$\{v,w\}$ are adjacent 2-token
configurations in~$F_2(G)$.)
We will use this in the figures ahead, because we find it a convenient and
compact way to present a coloring for~$F_2(G)$. 
Also, we will denote a 2-token configuration $\{u,v\}$ simply as~$uv$, to
simplify the notation, and we will call it a configuration instead of a vertex
of~$F_2(G)$ to avoid confusion with the vertices of~$G$.

When $G = K_n$, such an edge-coloring of~$K_n$ is exactly a proper
edge-coloring of~$K_n$. 
Hence $\chi(F_2(K_n)) = \chi'(K_n)$ and thus one direction of
Conjecture~\ref{conj:nossa} is already known: $\chi(F_2(K_n)) = n-1 =
\chi(G)-1$ for~$n$ even, because there is a partition of the edges of~$K_n$
into~$n-1$ perfect matchings, that is, $\chi'(K_n) = n-1$. 

A graph~$G$ is \defi{critical} if $\chi(G-v) = \chi(G)-1$ for every vertex~$v$
of~$G$.
We say~$G$ is \defi{$k$-critical} if~$G$ is critical and $\chi(G) = k$. 
We observe that a possible counter-example~$G$ for the other direction would have 
the following properties:
\begin{inparaenum}[(a)]
\item $\chi(G) > 3$;
\item $G$ is critical;
\item the maximum degree $\Delta(G) \geq \chi(G)$;
\item $\chi(G) > \clique(G)$.
\end{inparaenum}
Indeed, it is known that a graph~$G$ is bipartite if and only if~$F_2(G)$ is bipartite
(see~\cite{FabilaMonroyFHHUW2012}), which implies $(a)$.
If $G$ is non-critical, there exists a vertex $v \in V(G)$ such that $\chi(G) = \chi(G-v)$.
In this case, it suffices to consider the subgraph of~$F_2(G)$ consisting of all 
configurations with a token at~$v$, which is isomorphic to $F_1(G-v)$, and this in
turn is isomorphic to $G-v$.
We then have $\chi(F_2(G)) \ge \chi(G-v) = \chi(G)$.
This proves $(b)$.
Assume now that~$G$ is critical, so~$G$ must have minimum degree at least $\chi(G)-1$.
If $\Delta(G) < \chi(G)$, then~$G$ is a $(\chi(G)-1)$-regular graph, and, by Brook's Theorem
(see~\cite{Brooks}), $G$ must be either an odd cycle or a complete graph, which is not
the case by $(a)$ and given that $\chi(F_2(K_n)) = \chi'(K_n)=n$ for $n$ odd.
Thus, $(c)$ must hold.
Finally, suppose $\chi(G) \le \omega(G)$.
By Theorem~\ref{thm:cliques} we have $\omega(F_k(G)) = \min\{\omega(G), \max\{n-k+1,k+1\}\}$,
and then, as~$G$ is not a complete graph, we have $\chi(F_2(G)) \ge \omega(F_2(G)) = \omega(G) = \chi(G)$,
which proves $(d)$.

Next we prove that Conjecture~\ref{conj:nossa} holds for two classes of
4-critical graphs. 
A vertex in a graph~$G$ is \defi{universal} if it is adjacent to all other
vertices of~$G$. 

\begin{lemma}\label{lem:universalvx}
  For every graph~$G$ on $n \geq 5$ vertices, with $\chi(G) = 4$ and a
  universal vertex~$u$, $\chi(F_2(G)) = 4$. 
\end{lemma}
\begin{proof}
  It is enough to prove that there is no proper 3-coloring for~$F_2(G)$,
  because we already know that $\chi(F_2(G)) \leq
  \chi(G)$~\cite[Thm.~6]{FabilaMonroyFHHUW2012}.
  For a contradiction, suppose there is such a 3-coloring with colors 1, 2, 3.  
  Let $I_j = \{v \colon \mbox{configuration $uv$ has color $j$}\}$
  for~$j=1,2,3$. 
  Note that the sets $I_1,I_2,I_3$ form a partition of $V(G-u)$. 
  Figure~\ref{fig:coloring} presents the information we are deriving on this
  coloring. 
  
  \begin{figure}[h!]
    \centering
    \resizebox{0.55\textwidth}{!}{\includegraphics{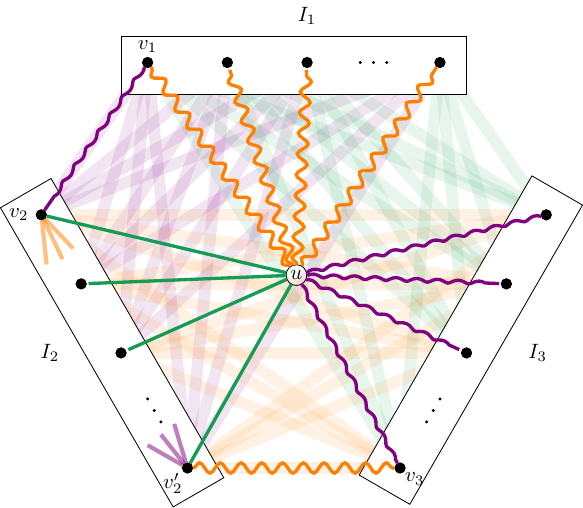}}
    \caption{The 3-coloring of $F_2(G)$ from Lemma~\ref{lem:universalvx}
    represented as an edge-coloring of the complete graph on $|V(G)|$ vertices.}
    \label{fig:coloring}
  \end{figure}

  Because this 3-coloring of~$F_2(G)$ is proper, there cannot be an edge of~$G$
  with the two ends in the same~$I_j$, that is, each~$I_j$ is an independent
  set in~$G$. 
  Moreover, since $\chi(G) = 4$, each~$I_j$ is non-empty.
  For every $v \in I_1$ and $w \in I_2$, the configuration~$vw$ must have
  color~3, because the configuration~$uv$ has color~1 and the
  configuration~$uw$ has color~2, and both are edges of~$G$. 
  Thus, all configurations~$vw$ for $v \in I_1$ and $w \in I_2$ have color~3. 
  Similarly, all configurations~$vw$ for $v \in I_1$ and~$w \in I_3$ have
  color~2, and all configurations~$vw$ for $v \in I_2$ and $w \in I_3$ have
  color~1. 

  Because $n \geq 5$, at least one of the sets~$I_j$ has at least two vertices. 
  Without loss of generality, say~$|I_2| \geq 2$.
  Because $\chi(G) = 4$, there must be an edge between a vertex $v_1 \in I_1$
  and a vertex~$v_2 \in I_2$.  
  For every $v \in I_2$ distinct from~$v_2$, as argued in the previous
  paragraph, the configurations $v_1v_2$ and~$v_1v$ have color~3. 
  Also, as both~$v_2$ and~$v$ are in~$I_2$, the configurations~$uv_2$ and~$uv$
  have color~2. 
  Thus, the configuration~$v_2v$ must have color~1.
  So all configurations~$v_2v$, for every $v \in I_2$ distinct from~$v_2$, have
  color~1. 
  However, there must also be an edge between a vertex $v_3 \in I_3$ and a
  vertex $v'_2 \in I_2$, and this implies that all configurations~$v'_2v$, for
  every $v \in I_2$ distinct from~$v'_2$, have color~3. 
  Because $|I_2| \geq 2$, we reach a contradiction
  (on the color of the configuration~$v_2v'_2$ if $v_2 \neq v'_2$, or, if~$v_2
  = v'_2$, on the color of the configurations~$v_2v$ for any $v \in I_2$
  distinct from~$v_2$; at least one such configuration exists because $|I_2|
  \geq 2$).  
\hfill$\qed$
\end{proof}

A well-known class of 4-critical graphs are the odd wheels $W_{2k+1}$. 
A consequence of Lemma~\ref{lem:universalvx} is that they do not contradict
Conjecture~\ref{conj:nossa}, because they have a universal vertex. 

\medskip 

Let~$G$ be a graph on~$n$ vertices $v_1,\ldots,v_n$. 
The \defi{Mycielski graph of~$G$} is the graph~$M(G)$ that contains a copy
of~$G$ together with~$n+1$ vertices $u_0,u_1,\ldots,u_n$, where each~$u_i$ is
adjacent to~$u_0$ and to each neighbor of~$v_i$ in~$G$, for $i=1,\ldots,n$. 
Mycielski~\cite{Mycielski1955} proved that $\chi(M(G)) = \chi(G)+1$.
It is not hard to prove that, for an odd~$k$, the graph $M(C_k)$ is 4-critical. 
Note that~$M(C_k)$ does not have a universal vertex, so
Lemma~\ref{lem:universalvx} does not apply to~$M(C_k)$.
Yet we can prove the following on~$F_2(M(C_k))$, which implies that~$M(C_k)$
for odd~$k$ does not contradict Conjecture~\ref{conj:nossa} either. 

\begin{lemma}\label{lem:Ck}
  For every odd $k$, $\chi(F_2(M(C_k))) = 4$. 
\end{lemma}

Before presenting the proof for this lemma, we prove a seemingly unrelated
result that will be used in the proof of the lemma. 
Let~$s$ be a string on a 3-letter alphabet. 
We say that~$s$ is \defi{special} if it has the form $AB^jC$ for some $j \geq
1$, where $A$, $B$, and $C$ are the three letters of the alphabet. 

\begin{claim2}\label{claim:ABC}
  Let~$s$ be a string on a 3-letter alphabet.
  If~$s$ contains the three letters of the alphabet, then~$s$ contains a
  special string as a substring. 
\end{claim2}
\begin{proof}
  The proof is by induction on the length of~$s$.
  If~$s$ has length~3, as~$s$ has the three letters of the alphabet, then~$s$
  is itself a special string with $j=1$.
  Otherwise, without loss of generality, suppose that~$s$ starts with the
  letter~$A$, and that~$B$ is the second letter different from~$A$ appearing
  in~$s$.  
  Call~$C$ the remaining letter of the alphabet. 
  If~$s$ starts with $A^iB^jC^k$, for $i,j,k \geq 1$, then we have the special
  substring~$AB^jC$ in~$s$ and we are done.
  Otherwise, $s$ starts with $A^iB^jA^k$, for $i,j,k \geq 1$. 
  Apply induction on the string~$s'$ obtained from~$s$ by removing the
  prefix~$A^i$.
  Note that~$s'$ still contains the three letters of the alphabet, as there is
  at least one~$A$ left in~$s'$ because $k \geq 1$. 
  Now, by induction, $s'$ contains a special substring, which is also a
  substring of~$s$, hence we are done. 
\hfill$\qed$
\end{proof}

\begin{proof}[of Lemma~\ref{lem:Ck}]
  We will prove that there is no proper 3-coloring for $F_2(M(C_k))$. 
  Let $C_k = (v_1,\ldots,v_k)$ and consider $u_0,u_1,\ldots,u_k$ as in the
  construction of~$M(C_k)$. 
  For a contradiction, suppose there is such a 3-coloring with colors 1, 2, 3.  
  Let $I_j = \{v \colon \mbox{configuration $u_0v$ has color $j$}\}$ for~$j=1,2,3$.
  The sets $I_1,I_2,I_3$ form a partition of $V(M(C_k)-u_0)$, and each~$I_j$ is
  an independent set in $M(C_k)-u_0$.
  Because $\chi(M(C_k)) = 4$ and $u_1,\ldots,u_k$ are the neighbors of~$u_0$,
  there is at least one~$u_i$ in each~$I_j$. 
  Furthermore, the partition $\{I_1,I_2,I_3\}$ induces a proper 3-coloring
  of~$C_k$, and so there is also at least one~$v_i$ in each~$I_j$. 

  Let~$s$ be the length-$k$ string on the alphabet $\{1,2,3\}$ consisting of
  the sequence of indices~$j$ such that $u_i \in I_j$ for~$i=1,\ldots,k$.
  As each~$I_j$ contains some~$u_i$, the three indices $1,2,3$ appear in~$s$. 
  We apply Claim~\ref{claim:ABC} to~$s$, deriving, without loss of generality,
  that $u_1 \in I_1$, $u_2,\ldots,u_\ell \in I_2$, and $u_{\ell+1} \in I_3$,
  for $\ell \geq 2$.
  Now let us derive a contradiction by analyzing two cases. 

  If $\ell = 2$, then the situation is depicted in Figure~\ref{fig:MCk}(a).
  Because~$u_0$ is adjacent to $u_1,u_2,u_3$, the configuration~$u_1u_2$ must
  have color~3 and the configuration $u_2u_3$ must have color~1. 
  Because~$v_2$ is adjacent to~$u_1$ and~$u_3$, the configuration~$u_0v_2$ must
  have color~2, which means $v_2 \in I_2$.
  But then the configuration~$u_2v_2$ cannot be colored with any of the colors
  1, 2, 3. 

  If $\ell > 2$, then the situation is depicted in Figure~\ref{fig:MCk}(b).
  Because~$v_2$ is adjacent to~$u_1$ and~$u_3$, the configuration~$u_0v_2$ must
  have color~3, which means $v_2 \in I_3$.
  Because~$u_0$ is adjacent to~$u_3$ and~$u_{\ell+1}$, the
  configuration~$u_3u_{\ell+1}$ must be of color~1, and the
  configuration~$u_{\ell+1}v_2$ must be of color~2. 
  But then the configuration $u_1u_{\ell+1}$ cannot be colored with any of the
  colors 1, 2, 3. 
\hfill$\qed$
\end{proof}

  \begin{figure}[h!]
    \centering
    \begin{minipage}{0.5\textwidth}
        \centering
        (a) \includegraphics[width=0.7\textwidth]{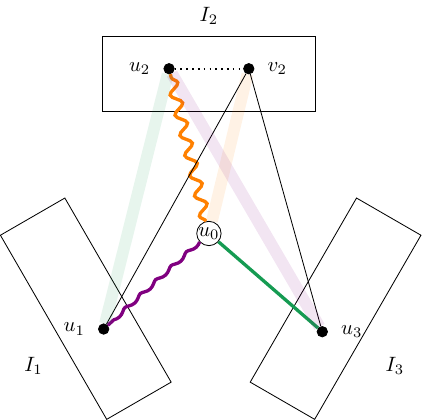}
    \end{minipage}\hfill
    \begin{minipage}{0.5\textwidth}
        \centering
        (b) \includegraphics[width=0.7\textwidth]{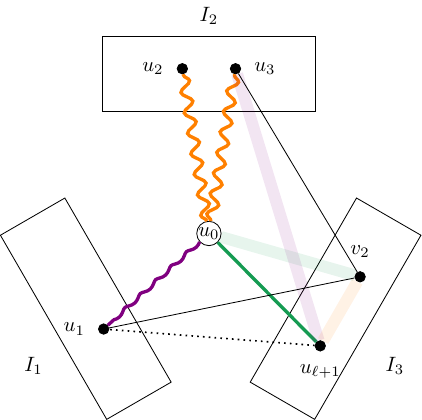} 
    \end{minipage}
    \caption{Cases for the proof of Lemma~\ref{lem:Ck}.} 
    \label{fig:MCk}
  \end{figure}

Let~$D$ be a strongly connected digraph of order~$n$.
We know that~$F_k(D)$ is also strongly connected and that $\dichi(F_1(D)) =
\dichi(D)$.  
One might ask whether $\dichi(F_k(D)) \leq \dichi(F_{k-1}(D))$ for $k > 1$. 
However this does not always hold.
For instance, for all $n \equiv 4 \pmod{6}$,
we have~$\chi(F_2(K_n)) = n-1$ and $\chi(F_3(K_n)) = n$~\cite{EtzionB1996}.
It is worth mentioning that we have found only one example of graph~$G$
different from the complete graph, such that $\chi(F_k(G)) < \chi(F_{k-1}(G))$
for some $1 < k \leq n/2$.
This graph, a $6$-critical one, can be obtained from the complete graph on~$8$
vertices by deleting a cycle of length~$5$.
For this specific graph, we have $\chi(F_k(G)) = 6$ for all $1 \leq k \leq 3$,
and $\chi(F_4(G)) = 5$.
For all other graphs that we considered, we observed $\chi(F_k(G)) =
\chi(F_{k-1}(G))$ for all $1 < k \leq n/2$.

\section{Final remarks}\label{sec:final}

In this paper we introduced a generalization of token graphs considering
digraphs. 
We prove several results relating properties of the digraph to properties of
its token digraphs.
This study is interesting on its own, but can also share light to the
undirected case.

We found particularly challenging the study of acyclic partitions and the
dichromatic number of token digraphs.
Our study led us to interesting questions not only for the directed case, but
also for the undirected case, including a conjecture stated in
Section~\ref{sec:acyclic_partition}.

\section*{Acknowledgements}

This work was initiated during the ChiPaGra workshop 2022 in Nazaré Paulista,
São Paulo, Brazil: we thank the organizers and the attendees for a great
atmosphere.
This work is supported by the following grants and projects: 
\begin{inparaenum}[(i)]
\item FAPESP-ANID Investigación Conjunta grant 2019/13364-7.
\item CGF and CNL were supported by Coordenação de Aperfeiçoamento de Pessoal
de Nível Superior -- Brasil -- CAPES -- Finance Code 001, and by the Conselho
Nacional de Desenvolvimento Científico e Tecnológico (CNPq) of Brazil grant
404315/2023-2.
\item CGF was supported by CNPq (grants 310979/2020-0 and 423833/2018-9).
\item CNL was supported by CNPq (grant 312026/2021-8) and by L'ORÉAL-UNESCO-ABC For Women In Science.
\item ATN was supported by ANID/Fondecyt Postdoctorado 3220838, and by ANID Basal Grant CMM FB210005. 
\item GS was supported by ANID Becas/Doctorado Nacional 21221049.
\item JPP was supported by ANID Becas/Doctorado Nacional 21211955.
\end{inparaenum}

\bibliographystyle{splncs04}
\bibliography{token_graphs}
    
\end{document}